\documentclass[12pt,a4paper,reqno]{amsart}

\usepackage{float}

\usepackage{anysize}

\marginsize{2.5cm}{2.5cm}{3.5cm}{3.5cm}

\usepackage[utf8]{inputenc}

\usepackage{amsmath}
\usepackage{amssymb}
\usepackage{array}

\usepackage{graphicx}
\usepackage{tikz}
\usepackage{color}

\newtheorem{theorem}{Theorem}[section]
\newtheorem{corollary}[theorem]{Corollary}
\newtheorem{lemma}[theorem]{Lemma}

\begin{document}

\title[The Mangoldt function and the zeros of $\zeta(s)$]{The Mangoldt function and the non-trivial zeros \\ of the Riemann zeta function}
\author{Jesús Guillera}
\address{Department of Mathematics, University of Zaragoza, 50009 Zaragoza, SPAIN}
\email{jguillera@gmail.com}
\date{}

\maketitle

\begin{abstract}
We prove a formula for the Mangoldt function which relates it to a sum over all the non-trivial zeros of the Riemann zeta function, in addition we analyze a truncated version of it.
\end{abstract}

\section{Notation}
We use the notation $\rho=\beta+i \gamma$ for the non-trivial zeros of the zeta function. Following Riemann, we define $\alpha=-i (\rho-\frac12)$. Observe that $\rho=1/2+i \alpha$ with ${\rm Re}(\alpha)=\gamma$ and ${\rm Im}(\alpha)=1/2-\beta$, and that the Riemann Hypothesis is the statement $\alpha={\rm Re}(\alpha)$. It is known that $0<\beta<1$ (critical band), and therefore $-1/2<{\rm Im}(\alpha)<1/2$. If we let $\mu=\frac12-\beta$ then $\alpha=\gamma+i \mu$, and $-1/2 < \mu < 1/2$. This notation simplifies the appearance of our formulas. As usual in Number Theory, $\log$ denotes the neperian logarithm. 

\section{Introduction}
In $1911$ E. Landau proved that for any fixed $t>1$
\begin{equation}\label{landau1}
\sum_{0<\gamma\leq T} t^{\rho} = \frac{-T}{2\pi} \Lambda(t) + \mathcal{O}(\log T),
\end{equation}
where $\rho$ runs over the non-trivial zeros of the Riemann zeta function $\zeta(s)$ and $\Lambda(t)$ is the Mangoldt function which is equal to $\log p$ if $x$ is a power of a prime number $p$ and $0$ otherwise. Since the use of (\ref{landau1}) is limited by its lack of uniformity in $t$, Gonek was interested in a version of it uniform in both variables and in  \cite{Gonek1,Gonek2}, he gives the remarkable formula
\[
\sum_{0<\gamma\leq T} t^{\rho} = \frac{-T}{2\pi} \Lambda(t) + E(t,T),
\]
where the error term $E(t,T)$ has the estimation
\[
E(t,T)=\mathcal{O} \left(t\log 2tT \log \log 3t \right)+\mathcal{O} \left(\log t \, {\rm min} ( T ; \frac{t}{\langle t \rangle} ) \right)  + \mathcal{O} \left(\log 2T \, {\rm min}  (T ; \frac{1}{\log t} ) \right),
\]
with $\langle t \rangle$ denoting the distance between $t$ and the nearest prime power other than $t$. Gonek's formula is also commented in \cite{Kalape}. The aim of this paper is to approximate $\Lambda(t)$ in a good way. Of course we can do it with the Landau-Gonek's formula:
\begin{equation}\label{Landau-Gonek}
\Lambda(t)=\frac{-2\pi}{T} \sqrt{t} \sum_{0< \gamma \leq T} \cos(\alpha \log t)
+\frac{E(t,T)}{T},
\end{equation}
where we have used the Riemann's notation $\alpha=-i (\rho-1/2)$. Observe that the formula (\ref{landau1}) or either (\ref{Landau-Gonek}) imply
\begin{equation}\label{Landau-limit}
\Lambda(t)=-2\pi \sqrt{t} \lim_{T \to +\infty} \frac{1}{T}  \sum_{0< \gamma \leq T} \cos(\alpha \log t).
\end{equation}
which has the surprising property that neglecting a finite number of zeros of zeta we still recover the Mangoldt's function. Also surprising are the self-replicating property of the zeros of zeta observed recently in the statistics of \cite{perez-marco}, and later proved in \cite{ford-zaha}; and the property of the zeros discovered by Y. Matiyasevich \cite{Matiya}. In this paper we will prove the new formula:
\[
\Lambda(t) = -4 \pi \sqrt{t} \cot \frac{x}{2} \sum_{\gamma >0} \frac{\sinh x \alpha}{\sinh \pi \alpha} \cos(\alpha \log t)+2\pi \cot \frac{x}{2} \left( t-\frac{1}{t^2-1} \right)+\varepsilon(t,x).
\]
and find bounds for the error term $\varepsilon(t,x)$. In addition, letting $\cot(x/2)=(\log T)/T$ we will prove that for integers $t>2$, the following truncated version of it holds
\begin{align*}
\Lambda(t) &= -4 \pi \sqrt{t} \left( \sum_{0< \gamma < T} \frac{\sinh x \alpha}{\sinh \pi \alpha} \cos(\alpha \log t) \right) \frac{\log T}{T} + 2\pi \left( t-\frac{1}{t^2-1} \right) \frac{\log T}{T} \\ &+ \mathcal{O} \left( t^2 (\log t) \frac{\log^2 T}{T^2} \right),
\end{align*}
and we also will get the estimation of the error for non-integers $t$. Finally, observing that
\[
\Lambda(t)=-4 \pi \sqrt{t} \lim_{x \to \pi^{-}} \left( \cot \frac{x}{2} \sum_{\gamma>0} \frac{\sinh x \alpha}{\sinh \pi \alpha} \cos(\alpha \log t) \right),
\]
we see that it shares with (\ref{Landau-limit}) the property of invariance when we neglect a finite number of zeros. In the last section we give the new function
\[
\Phi_2(t)=-\sum_{m=1}^{T} T^{-m/T} \frac{\Lambda(m)}{\sqrt{m}} \cos(t \log m) + C \, \sqrt{t},
\]
where $C \approx 0.12$ is a constant. This function has cusps at the non-trivial zeros of zeta. It looks like that this function is interesting and I will continue investigating it.

\section{Series involving the Mangoldt function}

The formulas that we will prove in this section involve the Mangoldt's function and a sum over the non-trivial zeros of the Riemann-zeta function.
\begin{theorem}\label{main-thm}
Let $\Omega = \mathbb{C}-(-\infty,0]$ (the plane with a cut along the real negative axis). We shall denote by $\log z$ the main branch of the $\log$ function defined on $\Omega$ taking $|\arg(z)|<\pi$. We also denote by $z^s=\exp(s\log(z))$, the usual branch of $z^s$ defined also on $\Omega$. For all $z \in \Omega$ we have
\begin{equation}\label{for-main-thm}
\sum_{n=1}^{\infty} \frac{\Lambda(n)z}{\pi \sqrt{n} (z+n)}-\sum_{n=1}^{\infty} \frac{\Lambda(n)}{\pi \sqrt{n} (1+nz)}=\sqrt{z}-\frac{\zeta'(\frac12)}{\pi \zeta (\frac12)}-2\sum_{\gamma>0} \frac{\sin (\alpha \log z) }{\sinh \pi \alpha} + h(z),
\end{equation}
where
\begin{equation}\label{h-of-z}
h(z)=\frac{1}{\sqrt{z}(z^2-1)}-\frac{1}{2z-2}+\frac{\log(8\pi)+C}{\pi} \frac{1}{z+1}-\frac{2}{\pi} \frac{\sqrt{z}}{z+1}\arctan \frac{1}{\sqrt{z}}.
\end{equation}
\end{theorem}

\begin{proof}
We consider the function
\[
f(s)=\frac{\zeta'(s+\frac12)}{\zeta(s+\frac12)} \frac{\pi}{\sin \pi s} z^s,
\]
and let $I_0$, $I_{r}$, $I_{\ell}$,  where $I_{r}=I_1+I_2+I_3$ and $I_{\ell}=I_4+I_5+I_6$, be the analytic continuation of the integral
\[
I=\frac{1}{2\pi i} \int f(s) ds,
\]
along the indicated sides of the contour of the figure. It is a known result that all the zeros of $\zeta(s+1/2)$ are in the band among the lines red and green. 
\begin{center}
\begin{tikzpicture}
\draw[thick,color=blue] (1,0) rectangle (11,5);
\filldraw[color=black,fill=black] (7.25,2.5) circle (0.05);
\filldraw[color=black,fill=black] (1,0) circle (0.05);
\filldraw[color=black,fill=black] (1,5) circle (0.05);
\filldraw[color=black,fill=black] (11,0) circle (0.05);
\filldraw[color=black,fill=black] (11,5) circle (0.05);
\filldraw[color=black,fill=black] (6.5,0) circle (0.05);
\filldraw[color=black,fill=black] (6.5,5) circle (0.05);
\draw[red, thick] (6.5,0) -- (6.5,5);
\draw[->,red, very thick] (6.5,2.5) -- (6.5,2.5);
\draw[->,red, very thick] (3.74,0) -- (3.75,0);
\draw[->,red, very thick] (3.75,5) -- (3.74,5);
\draw[->,red, very thick] (1,2.6) -- (1,2.5);
\draw[->,red, very thick] (8.75,5) -- (8.76,5);
\draw[->,red, very thick] (8.75,0) -- (8.74,0);
\draw[->,red, very thick] (11,2.6) -- (11,2.5);
\draw[green,thick] (8,0) -- (8,5);
\node (n1) at (1,-0.5) {$-\infty-iT$};
\node (n2) at (1,5.5) {$-\infty +iT$};
\node (n3) at (11,-0.5) {$+\infty-iT$};
\node (n4) at (11,5.5) {$+\infty +iT$};
\node (n5) at (6.5,-0.5) {$-\frac12-iT$};
\node (n6) at (6.5,5.5) {$-\frac12+iT$};
\node (n7) at (9,2.5) {$|z|<1$};
\node (n8) at (4,2.5) {$|z|>1$};
\node (ne1) at (6.35, 1.5) {{\tiny Integrals extended to all $z$ by analytic continuation}};
\node (n9) at (8.75,-0.5) {$I_3$};
\node (n10) at (8.75,5.5) {$I_1$};
\node (n11) at (3.75,-0.5) {$I_6$};
\node (n12) at (3.75,5.5) {$I_4$};
\node (n13) at (11.4,2.5) {$I_2$};
\node (n14) at (6.1,2.5) {$I_0$};
\node (n15) at (0.65,2.5) {$I_5$};
\end{tikzpicture}
\end{center}
We will follow this scheme of proof: The integral along the line $\sigma=-1/2$ is calculated for $|z|>1$ integrating to the left and for $|z|<1$ integrating to the right. Both expressions are different but valid for $z \in \Omega$ by analytic continuation. Finally, equating both expressions we will arrive at (\ref{for-main-thm}). 
\par Indeed, if $|z|<1$, integrating to the right hand side, we get by applying the residues theorem that
\begin{align}
& I_0+ I_{r}=-{\rm res}_{s=\frac12} \left( \frac{\zeta'(s+\frac12)}{\zeta(s+\frac12)} \frac{\pi}{\sin \pi s} z^s \right)-\sum_{n=0}^{\infty} {\rm res}_{s=n} \left( \frac{\zeta'(s+\frac12)}{\zeta(s+\frac12)} \frac{\pi}{\sin \pi s} z^s \right) -  \nonumber  \\
& \sum_{|\gamma|<T} {\rm res}_{s=\rho-\frac12} \left( \frac{\zeta'(s+\frac12)}{\zeta(s+\frac12)} \frac{\pi}{\sin \pi s} z^s \right) \!=\! \pi \sqrt{z}-\sum_{n=0}^{\infty} (-1)^n \frac{\zeta'(n+\frac12)}{\zeta(n+\frac12)}z^n-\pi\sum_{|\gamma| < T} \frac{z^{\rho-\frac12}}{\sin \pi (\rho-\frac12)}. \nonumber
\end{align}
Hence, by analytic continuation, we have that for all $z \in \Omega$
\begin{equation}\label{int-right}
I_0+I_{r} = \pi \sqrt{z}-\frac{\zeta'(\frac12)}{\zeta(\frac12)}-\sum_{n=1}^{\infty} \frac{\Lambda(n) z}{\sqrt{n}(z+n)} - \pi \sum_{|\gamma| < T} \frac{z^{\rho-\frac12}}{\sin \pi (\rho-\frac12)}
\end{equation}
If $|z|>1$, then following the way to the left hand side, we deduce that
\begin{align}
I_0+I_{\ell} &= \sum_{n=1}^{\infty} {\rm res}_{s=-2n-\frac12} \left( \frac{\zeta'(s+\frac12)}{\zeta(s+\frac12)} \frac{\pi}{\sin \pi s} z^s \right) + \sum_{n=1}^{\infty} {\rm res}_{s=-n} \left( \frac{\zeta'(s+\frac12)}{\zeta(s+\frac12)} \frac{\pi}{\sin \pi s} z^s \right) \nonumber \\
&= \sum_{n=1}^{\infty} \frac{\zeta'(-2n)}{\zeta(-2n)} \sin(2 \pi n) z^{-2n-\frac12}+\sum_{n=1}^{\infty} (-1)^n \frac{\zeta'(\frac12-n)}{\zeta(\frac12-n)}z^{-n}, \label{lastsums}
\end{align}
where we understand the expression inside the first sum of (\ref{lastsums}) as a limit based on the identity
\[ 
\lim_{s \to -2n} (s+2n) \frac{\zeta'(s)}{\zeta(s)} = \lim_{s \to -2n} \frac{\zeta'(s)}{\zeta(s)}  \frac{\pi (s+2n)}{\sin \pi s} \frac{\sin \pi s}{\pi} = \lim_{s \to -2n} \frac{\sin{\pi s}}{\pi} \frac{\zeta'(s)}{\zeta(s)}.
\] 
We use the functional equation (which comes easily from the functional equation of $\zeta(s)$):
\begin{equation}\label{zpz}
\frac{\zeta'(1-s)}{\zeta(1-s)}=\log 2\pi-\psi(s)+\frac{\pi}{2}\tan \frac{\pi s}{2}-\frac{\zeta'(s)}{\zeta(s)}.
\end{equation}
to simplify the sums in (\ref{lastsums}). For the first sum in (\ref{lastsums}), we obtain 
\[
\sum_{n=1}^{\infty} \frac{\zeta'(-2n)}{\zeta(-2n)} \sin(2\pi n) z^{-2n-\frac12}=\pi \sum_{n=1}^{\infty} z^{-2n-\frac12},
\]
and for the last sum in (\ref{lastsums}), we have
\[
\log 2 \pi \sum_{n=1}^{\infty} (-1)^n z^{-n}-\sum_{n=1}^{\infty} (-1)^n \psi \left(\frac12+n \right)z^{-n}+\frac{\pi}{2}\sum_{n=1}^{\infty} z^{-n}-\sum_{n=1}^{\infty} (-1)^n \frac{\zeta'(n+\frac12)}{\zeta(n+\frac12)}z^{-n},
\]
where $\psi$ is the digamma function, which satisfies the property
\[
\psi \left(\frac12+n \right) = 2 h_n - C - 2 \log 2, \qquad h_n = \sum_{j=1}^n \frac{1}{2j-1}.
\]
Using the identity, due to Hongwei Chen \cite[p.299, exercise 34]{bailey-et-al}
\[
2 \sum_{n=1}^{\infty} (-1)^n h_n z^{-n} = i \frac{\sqrt{z}}{z+1} \log \frac{\sqrt{z}+i}{\sqrt{z}-i}=-2 \frac{\sqrt{z}}{z+1} \arctan \frac{1}{\sqrt{z}},
\]
we get that for $|z|>1$
\[
I_0+I_{\ell} = - \frac{\pi}{\sqrt{z}(z^2-1)}-\frac{\log 2 \pi}{z+1}-\frac{C+\log 4}{z+1}+\frac{\pi}{2z-2} +2 \frac{\sqrt{z}}{z+1} \arctan \frac{1}{\sqrt{z}}-\sum_{n=1}^{\infty} (-1)^n \frac{\zeta'(n+\frac12)}{\zeta(n+\frac12)}z^{-n}
\]
Then, by analytic continuation, we obtain that for all $z \in \Omega$:
\begin{equation}\label{int-left}
I_0+I_{\ell} =-\frac{\pi}{\sqrt{z}(z^2-1)}-\frac{C+\log 8\pi}{z+1}+\frac{\pi}{2z-2} + 2 \frac{\sqrt{z}}{z+1} \arctan \frac{1}{\sqrt{z}}-\sum_{n=1}^{\infty} \frac{\Lambda(n)}{\sqrt{n}(1+zn)}.
\end{equation}
It is easy to deduce that $I_2=I_5=0$, and we will prove that $I_{r}$ and $I_{\ell}$ tend to $0$ as $T \to \infty$ in the Section \ref{sec-bounds} of this paper. Hence, by identifying (\ref{int-right}) and (\ref{int-left}), and observing that the pole at $z=1$ is removable, we complete the proof.
\end{proof}

\begin{theorem}
The following identity
\begin{equation}\label{main2-thm}
\sum_{\gamma>0} \frac{\sinh z \alpha}{\sinh \pi \alpha} - \sum_{n=1}^{\infty} \frac{\Lambda(n)}{2\pi \sqrt{n}} \left(\frac{i e^{iz}}{e^{iz}+n}-\frac{i e^{-iz}}{e^{-iz}+n} \right) = f(z),
\end{equation}
where
\[
f(z)=\sin \frac{z}{2} - \frac{1}{8}\tan \frac{z}{4} -\frac{C+\log 8\pi}{4\pi}\tan \frac{z}{2} - \frac{1}{4\pi \cos \frac{z}{2}} \log \frac{1-\tan \frac{z}{4}}{1+\tan \frac{z}{4}},
\]
holds for $|{\rm Re}(z)|<\pi$. 
\end{theorem}

\begin{proof}
Let 
\[
H(z)=\sqrt{z}-\frac{\zeta'(\frac12)}{\pi \zeta(\frac12)}+h(z).
\]
That is
\[
H(z)=\sqrt{z}-\frac{\zeta'(\frac12)}{\pi \zeta(\frac12)}+\frac{1}{\sqrt{z}(z^2-1)}-\frac{1}{2z-2}+\frac{\log(8\pi)+C}{\pi} \frac{1}{z+1}+\frac{i}{\pi} \frac{\sqrt{z}}{z+1}\log \frac{\sqrt{z}+i}{\sqrt{z}-i}.
\]
From (\ref{for-main-thm}), we see that the function $H(z)$ has the property $H(z)=-H(z^{-1})$. Hence
\begin{align}
H(z)=\frac{H(z)-H(z^{-1})}{2} =& \frac12 \left( \sqrt{z}-\frac{1}{\sqrt{z}} \right)
+ \frac12 \frac{1}{z^2-1} \left( \frac{1}{\sqrt{z}}+z^2 \sqrt{z} \right) - \frac14 \frac{z+1}{z-1} \nonumber \\ &- \frac{\log 8 \pi +C}{2\pi} \frac{z-1}{z+1} + \frac{i}{\pi} \frac{\sqrt{z}}{z+1} 
\log \frac{i\sqrt{z}-1}{i\sqrt{z}+1}+\frac12 \frac{\sqrt{z}}{z+1}. \label{H-of-z}
\end{align}
When $|{\rm Re}(z)|<\pi$ we have $e^{iz} \in \Omega$ so that, we may put $e^{iz}$ instead of $z$ in Theorem \ref{main-thm}. If in addition we multiply by $-i/2$, we get
\[
\sum_{\gamma>0} \frac{\sinh z \alpha}{\sinh \pi \alpha} - \sum_{n=1}^{\infty} \frac{\Lambda(n)}{2 \pi \sqrt{n}} \left(\frac{i e^{iz}}{e^{iz}+n}-\frac{i e^{-iz}}{e^{-iz}+n} \right)=-\frac{i}{2} H(e^{iz}).
\]
From (\ref{H-of-z}), we have
\begin{align}
-\frac{i}{2}H(e^{iz}) =& -\frac{i}{4} \left( e^{iz/2}-e^{-iz/2} \right) 
-\frac{i}{4} \left( \frac{e^{-iz/2}}{e^{2iz}-1}-\frac{e^{iz/2}}{e^{-2iz}-1}  \right)
+\frac{i}{8} \frac{e^{iz}+1}{e^{iz}-1} \nonumber \\
&+\frac{i}{4} \frac{\log 8\pi + C}{\pi} \frac{e^{iz}-1}{e^{iz}+1}+
\frac{1}{2\pi} \frac{e^{iz/2}}{e^{iz}+1} \log \frac{e^{i\frac{z+\pi}{2}}-1}{e^{i \frac{z+\pi}{2}}+1}
-\frac{i}{4} \frac{e^{iz/2}}{e^{iz}+1},
\nonumber
\end{align}
which we can write as
\begin{align}
-\frac{i}{2}H(e^{i z}) =& -\frac{i}{4} \left( e^{i z/2}-e^{-i z/2} \right) 
-\frac{i}{4} \frac{e^{3i z/2}+e^{-3i z/2}}{e^{i z}-e^{-i z}} \nonumber \\
&+\frac{i}{8} \frac{e^{i z/2}+e^{-i z/2}}{e^{i z/2}-e^{-i z/2}} 
+\frac{i}{4} \frac{\log 8\pi + C}{\pi} \frac{e^{i z/2}-e^{-i z/2}}{e^{i z/2}+e^{-i z/2}}  \nonumber \\
&+\frac{1}{2\pi} \frac{1}{e^{i z/2}+e^{-iz/2}} \log \frac{e^{i\frac{z+\pi}{4}}-e^{-i\frac{z+\pi}{4}}}{e^{i \frac{z+\pi}{4}}+e^{-i\frac{z+\pi}{4}}}-\frac{i}{4} \frac{1}{e^{iz/2}+e^{-iz/2}},
\nonumber
\end{align}
which simplifies to
\begin{align}
-\frac{i}{2}H(e^{iz}) =& \frac12 \sin \frac{z}{2}-\frac14 \frac{\cos \left(z+\frac{z}{2}\right)}{\sin z}+\frac18 \cot \frac{z}{2}-\frac{\log 8\pi + C}{4\pi} \tan \frac{z}{2} \nonumber \\ &+ \frac{1}{4\pi} \frac{1}{\cos \frac{z}{2}} \log \left(i \tan \frac{z+\pi}{4} \right)-\frac{i}{8} \frac{1}{\cos \frac{z}{2}}. 
\nonumber
\end{align}
As
\[
\frac{1}{4\pi} \frac{1}{\cos \frac{z}{2}} \log \left(i \tan \frac{z+\pi}{4} \right)-\frac{i}{8} \frac{1}{\cos \frac{z}{2}} = \frac{1}{4\pi} \frac{1}{\cos \frac{z}{2}} \log \tan \frac{z+\pi}{4},
\]
and using elementary trigonometric formulas we arrive at (\ref{main2-thm}).
\end{proof}

\section{New formulas for the Mangoldt function}
In this section we relate the Mangoldt's function to a sum over all the non-trivial zeros of the Riemann-zeta function and find bounds of the error term.

\begin{theorem}\label{theor-mang}
If $x \in [0, \pi)$ and $t>1$, then
\begin{align*}
& - 4\pi \sqrt{t} \cot\frac{x}{2} \sum_{\gamma>0} \frac{\sinh x \alpha}{\sinh \pi \alpha} \cos(\alpha \log t) + 4\pi \sqrt{t} \, g(x,t) \cot \frac{x}{2}  \\ & = 4 t \sqrt{t} \sum_{n=1}^{\infty}  \frac{\Lambda(n) \sqrt{n} \, \cos^2 \frac{x}{2}}{(t-n)^2 + 4nt \cos^2 \frac{x}{2}} - 4 t \sqrt{t} \sum_{n=1}^{\infty} \frac{\Lambda(n) \sqrt{n} \, \cos^2 \frac{x}{2}}{(nt-1)^2 + 4nt \cos^2 \frac{x}{2}},
\end{align*}
where $g(x,t)$ is the function 
\begin{multline}\label{fun-g}
g(x,t)=\frac{(1+t)\sin\frac{x}{2}}{2\sqrt{t}}-\frac{\sqrt{t}\sin \frac{x}{2}}{8\sqrt{t}\cos \frac{x}{2}+4(1+t)}-\frac{t\sin x (C+\log 8\pi)}{2\pi(1+t^2+2t \cos x)} - \\ 
\frac{(1+t)\sqrt{t} \cos \frac{x}{2}}{4\pi(1+t^2+2t \cos x)} \, \log \frac{1+t-2\sqrt{t}\sin \frac{x}{2}}{1+t+2\sqrt{t}\sin \frac{x}{2}} - \frac{(t-1)\sqrt{t} \sin \frac{x}{2}}{2\pi(1+t^2+2t \cos x)} 
\arctan \frac{t-1}{2\sqrt{t}\cos\frac{x}{2}}.
\end{multline}
\end{theorem}

\begin{proof}
Replace $z$ with $x-i\log t$ and take real parts. The function $g(x,t)$ is the real part of $f(x-i\log t)$.
\end{proof}
It is interesting to expand $g(x,t)$ in powers of $\pi-x$, and we get
\begin{align}\label{simply-g}
g(x,t) &=\frac12 \left(\frac{t+1}{t} - \frac{t}{t^2-1} \right) \sqrt{t} \nonumber \\ &+
\left( \frac{-t}{4(1+t)^2} + \frac{t(C+\log 8\pi)}{2\pi(t-1)^2} - \frac{t}{2\pi(t-1)^2}+
\frac{(1+t)\sqrt{t}}{4\pi(t-1)^2} \log \frac{\sqrt{t}-1}{\sqrt{t}+1}. \right)(\pi-x) \nonumber \\ & \qquad +\mathcal{O}(\pi-x)^2=\frac12 \left(\frac{t+1}{t} - \frac{t}{t^2-1} \right) \sqrt{t}+
\mathcal{O}\left( \frac{\pi-x}{t} \right),
\end{align}
which shown that $g(x,t)$ tends to a simple function as $x \to \pi^{-}$.

\begin{theorem}\label{mang-teor}
If $x \in [0, \pi)$, then
\begin{align}\label{mang-bound}
0 & < \left( -4\pi \sqrt{t} \cot\frac{x}{2} \sum_{\gamma>0} \frac{\sinh x \alpha}{\sinh \pi \alpha} \cos(\alpha \log t) + 4\pi \sqrt{t} \, g(x,t) \cot \frac{x}{2} \right) - \Lambda(t) \nonumber \\ & < F(t) + 4 \cos^2 \frac{x}{2} \left(3 t^2\log t + \frac{\pi^2}{2} t  + \frac14 \log t + 0.6 \right).
\end{align}
where $g(x,t)$ is the function (\ref{fun-g}), and 
\[
F(t) = E(t) \cdot 4 t \sqrt{t} \frac{\Lambda(\lfloor t \rfloor) \sqrt{\lfloor t \rfloor} \cos^2 \frac{x}{2} }{ (\{t\})^2 + 4 t \lfloor t \rfloor \cos^2 \frac{x}{2}} + 4 t \sqrt{t} \frac{\Lambda(\lfloor t \rfloor + 1) \sqrt{\lfloor t \rfloor + 1} \cos^2 \frac{x}{2} }{ (1-\{t\})^2 + 4 t (\lfloor t \rfloor + 1)  \cos^2 \frac{x}{2}},
\]
where $E(t)=0$ if $t$ is and integer and $1$ otherwise.
\end{theorem}

\begin{proof}
Let
\[
S=4 t \sqrt{t} \sum_{n=1}^{\infty}  \frac{\Lambda(n) \sqrt{n} \, \cos^2 \frac{x}{2}}{(t-n)^2 + 4nt \cos^2 \frac{x}{2}} - 4 t \sqrt{t} \sum_{n=1}^{\infty} \frac{\Lambda(n) \sqrt{n} \, \cos^2 \frac{x}{2}}{(nt-1)^2 + 4nt \cos^2 \frac{x}{2}}.
\]
First, we see that
\[
S < 4 t \sqrt{t} \sum_{n=1}^{\infty}  \frac{\Lambda(n) \sqrt{n} \, \cos^2 \frac{x}{2}}{(t-n)^2 + 4nt \cos^2 \frac{x}{2}}.
\]
The contribution of the values $n=\lfloor t \rfloor$ and $n=\lfloor t \rfloor +1$ to the above summation is equal to $\Lambda(t)+F(t)$, and the contribution of $n=\lfloor t \rfloor - 1$ is bounded by $4 t^2 \log t \, \cos^2 \frac{x}{2}$. Hence 
\[
S < \Lambda(t) + F(t) + 4 \left(t^2 \log t +
t \sqrt{t} \sum_{n=2}^{\lfloor t \rfloor -2} \frac{\Lambda(n) \sqrt{n}}{(t-n)^2} + t \sqrt{t} \sum_{n=\lfloor t \rfloor +2}^{\infty} \frac{\Lambda(n) \sqrt{n}}{(t-n)^2} \right) \cos^2 \frac{x}{2}.
\]
Then, as the Mangoldt function is bounded by the logarithm, we obtain
\[
S < \Lambda(t) + F(t) + 4 \left(t^2 \log t +
t \sqrt{t} \sum_{n=1}^{\lfloor t \rfloor -2} \frac{\log(n) \sqrt{n}}{(t-n)^2} + t \sqrt{t} \sum_{n=\lfloor t \rfloor +2}^{\infty} \frac{\log(n) \sqrt{n}}{(t-n)^2} \right) \cos^2 \frac{x}{2}.
\]
Then we can deduce that
\[
S < \Lambda(t) + F(t) + 4 \left(t^2 \log t + t \sqrt{t} \int_{1}^{t-1} \frac{\log(u)\sqrt{u}}{(t-u)^2} \, du + t \sqrt{t} \int_{t+1}^{\infty} \frac{\log(u)\sqrt{u} }{(t-u)^2}\, du \right) \cos^2 \frac{x}{2},
\]
by observing that the integrands are increasing and decreasing functions of $u$ respectively. With the help of Maple, we get 
\begin{align}
\sqrt{t} & \int_{1}^{t-1} \frac{\log(u)\sqrt{u}}{(t-u)^2} \, du  = \sqrt{t} \, \sqrt{t-1} \log(t-1) + \frac12 \log(t-1) \log \frac{\sqrt{t}-\sqrt{t-1}}{\sqrt{t}+\sqrt{t-1}}  \nonumber \\
& \, + \log \left( 1+\frac{1}{\sqrt{t}} \right)- \log \left( 1-\frac{1}{\sqrt{t}} \right) + 
\log \left( 1-\sqrt{1-\frac{1}{t}} \right) - \log \left( 1+\sqrt{1-\frac{1}{t}} \right) \label{integral-1} \\
& \, + {\rm dilog} \left( 1+\frac{1}{\sqrt{t}} \right) - {\rm dilog} \left( 1- \frac{1}{\sqrt{t}}\right)+
{\rm dilog} \left( 1-\sqrt{1-\frac{1}{t}} \right) - {\rm dilog} \left( 1 + \sqrt{1-\frac{1}{t}} \right), \nonumber
\end{align}
and
\begin{align}
\sqrt{t} & \int_{t+1}^{\infty} \frac{\log(u)\sqrt{u} }{(t-u)^2}\, du  =    \sqrt{t} \, \sqrt{t+1} \log(t+1) + \frac14 \log^2 t - \frac14 \log t \log(t+1) + \frac{\pi^2}{3} \nonumber \\
& \, + \frac12 \log(t+1) \log(\sqrt{t}+\sqrt{t+1}) - \frac12 (\log t) \log(\sqrt{t+1}-\sqrt{t}) + \log \frac{\sqrt{t+1}+\sqrt{t}}{\sqrt{t+1}-\sqrt{t}} \label{integral-2} \\
& \, + {\rm dilog} \left( 1+\sqrt{1+\frac{1}{t}} \right) + {\rm dilog} \left( \sqrt{1+\frac{1}{t}} \right), \nonumber
\end{align}
where $\rm dilog$ denotes the dilogarithm. Finally, by expanding asymptotically and bounding each of the terms of (\ref{integral-1}) and (\ref{integral-2}), we can derive that 
\[
t \sqrt{t} \int_{1}^{t-1} \frac{\log(u)\sqrt{u}}{(t-u)^2} \, du + t \sqrt{t} \int_{t+1}^{\infty} \frac{\log(u)\sqrt{u} }{(t-u)^2}\, du  =  2t^2\log t +\frac{\pi^2}{2}t +\frac14 \log t + h(t),
\]
where $h(t)$ is a positive decreasing function. Therefore $h(t)<h(2)<0.6$ for $t > 2$. 
\end{proof}

\begin{corollary}\label{mang-cotas}
If $x \in [0, \pi)$ and $t \geq 2$ is an integer, then
\begin{align}\label{mang-integer}
0 & < \left( -4\pi \sqrt{t} \cot\frac{x}{2} \sum_{\gamma>0} \frac{\sinh x \alpha}{\sinh \pi \alpha} \cos(\alpha \log t) + 4\pi \sqrt{t} \, g(x,t) \cot \frac{x}{2} \right) - \Lambda(t) \nonumber \\ & < 4 \cos^2 \frac{x}{2} \left(4 t^2\log t + \frac{\pi^2}{2} t  + \frac14 \log t + 0.6 \right).
\end{align}
where $g(x,t)$ is the function (\ref{fun-g}), 
\end{corollary}

\begin{lemma}\label{lemma-T}
If $x$ and $T$ are related by
\[
\cot \frac{x}{2} = \frac{\log T}{T},
\]
then for $T \geq 2$, we have
\begin{equation}\label{cota}
\left| \sum_{\gamma \geq T} \frac{\sinh x \alpha}{\sinh \pi \alpha} \cos(\alpha \log t) \right| < 3 \sqrt{t} \frac{2+\log T}{T}.
\end{equation}
\end{lemma}
\begin{proof}
Let $\eta={\rm Im}(\alpha)$. It is well known that $-1/2<\eta<1/2$ (critical band). As $\alpha=\gamma+i \eta$, we see that $|\cos(\alpha \log t) | < \cosh(|\eta| \log t) +\sinh(|\eta| \log t)$, and we get
\begin{align*}
\left| \sum_{\gamma \geq T} \frac{\sinh x \alpha}{\sinh \pi \alpha} \cos(\alpha \log t) \right|
& \leq \sum_{\gamma \geq T} \frac{\sinh x \gamma}{\sinh \pi \gamma} | \cos (\alpha \log t) | \leq \sum_{\gamma \geq T} \frac{\sinh x \gamma}{\sinh \pi \gamma} t^{|\eta|} \\
& \leq \sqrt{t} \sum_{\gamma \geq T} \frac{\sinh x \gamma}{\sinh \pi \gamma} \leq \sqrt{t} \sum_{\gamma \geq T} e^{-(\pi-x)\gamma}.
\end{align*}
As $x$ and $T$ are related by
\[
x=2 \, {\rm arccot} \frac{\log T}{T},
\]
we see that
\[
\pi-x=\frac{2 \log T}{T} - \frac23 \frac{\log^3 T}{T^3} + \mathcal{O} \left( \frac{1}{T^4} \right).
\]
Hence
\[
\left| \sum_{\gamma \geq T} \frac{\sinh x \alpha}{\sinh \pi \alpha} \cos(\alpha \log t) \right| < \sqrt{t} \sum_{\gamma \geq T} \exp{\frac{-2 \gamma \log T}{T}}.
\]
We subdivide the interval into intervals of length $1$. Hence, the left hand side is also less or equal that
\[
\sqrt{t} \left( \sum_{\gamma \in [T, T+1]} \exp{\frac{-2 \gamma \log T}{T}} +
\sum_{\gamma \in [T+1, T+2]} \exp{\frac{-2 \gamma \log (T+1)}{T+1}} + \cdots \right).
\]
From \cite[Corollary 1]{trudgian} we get that for $T \geq 2$ the number of zeros in an interval $[T, T+1]$ is less than $3 \log T$. Hence
\begin{align*}
\left| \sum_{\gamma \geq T} \frac{\sinh x \alpha}{\sinh \pi \alpha} \cos(\alpha \log t) \right| &< \sqrt{t} \sum_{n=T}^{\infty} 3 (\log n) \exp(-2 \log n) 
\\ &< 3 \sqrt{t}  \int_{T-1}^{+\infty} \frac{\log u}{u^2} du < 3 \sqrt{t} \frac{2+\log T}{T},
\end{align*}
which is the stated bound.
\end{proof}

\begin{corollary}
If $T \geq 2$ and $t \geq 2$ is a positive integer number, then
\begin{align}
& \left| -4\pi \sqrt{t} \left(\sum_{\gamma < T} \frac{\sinh x \alpha}{\sinh \pi \alpha} \cos(\alpha \log t) \right) \frac{\log T}{T} + 2\pi \left(t-\frac{1}{t^2-1}\right) \frac{\log T}{T} -\Lambda(t) \right| \nonumber \\ & <  4 \left(4 t^2\log t + \frac{\pi^2}{2} t  + 3\pi t + \frac14 \log t + 0.6 \right) \frac{\log^2 T}{T^2} + 24\pi t \frac{\log T}{T^2}. \label{coro-for3}
\end{align}
\end{corollary}

\begin{proof}
It is a consequence of the Corollary \ref{mang-cotas} and Lemma \ref{lemma-T}.
\end{proof}

\section{Graphics}
We have proved the following good approximation of the Mangoldt's function:
\begin{equation}\label{mang-approx}
\Lambda(t) \approx 4 \pi \sqrt{t} \left( \sum_{0< \gamma < T} T^{-2\gamma/T} \cos(\gamma \log t) \frac{t^{\mu}+t^{-\mu}}{2} \right) \frac{\log T}{T}  + 2 \pi \left(t-\frac{1}{t^2-1} \right) \frac{\log T}{T},
\end{equation}
where $\mu=1/2-\beta$, so $-1/2<\mu<1/2$. We use Sagemath \cite{stein} to draw the graphics. In Figure \ref{mangoldt} we see the graphic obtained with the formula (\ref{mang-approx}) summing over the $10000$ first non-trivial zeros of zeta, that is taking $T=9877.782654004$. The following estimations
\[
\varepsilon(t,T)=\mathcal{O} \left( t^2 \log t \frac{\log^2 T}{T^2} \right),
\qquad \varepsilon(t,T)=\mathcal{O} \left( \frac{t\log 2tT \log \log 3t}{T} \right),
\]
are respectively the errors that we get in the Mangoldt's function for integers $t>1$ if we use either our formula or either the Landau’s formula.

\begin{figure}[H]
\caption{Mangoldt}
\includegraphics[scale=0.75]{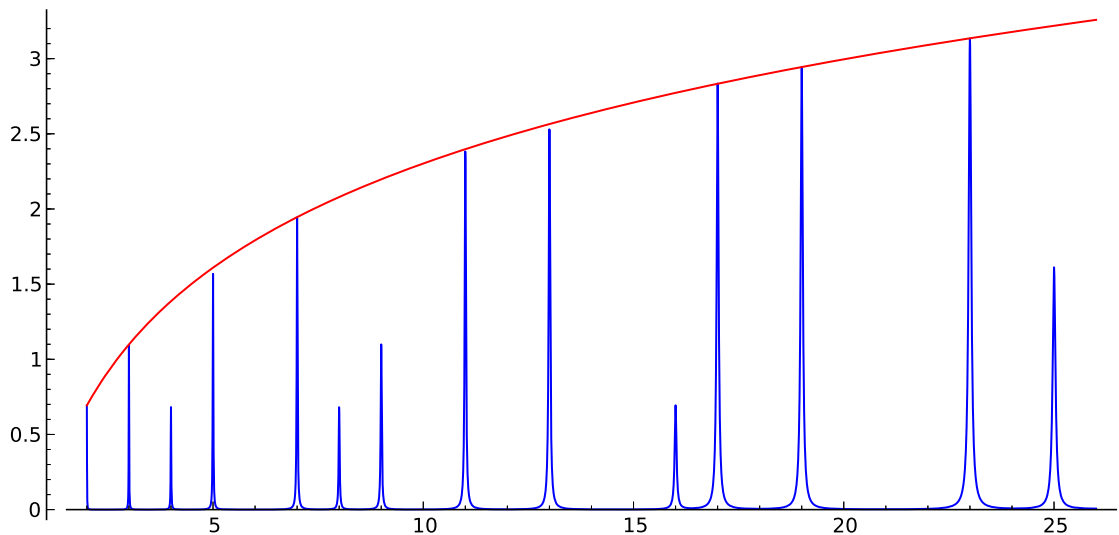}
\label{mangoldt}
\end{figure}
In this figure we have represented the function $\log(t)$ with the color red and the Mangoldt's function $\Lambda(t)$ with color blue.

\section{Another bound}\label{sec-bounds}
In this section we get another bound for
\[
\left| \sum_{\gamma \geq T} \frac{\sinh x \alpha}{\sinh \pi \alpha} \cos(\alpha \log t) \right|,
\]
From (\ref{int-right}) and (\ref{int-left}), we get
\begin{equation}\label{greaterT}
\sum_{\rho} \frac{z^{\rho-\frac12}}{\sin \pi (\rho-\frac12)}-\sum_{|\gamma| < T} \frac{z^{\rho-\frac12}}{\sin \pi (\rho-\frac12)} = \sum_{|\gamma| \geq T} \frac{z^{\rho-\frac12}}{\sin \pi (\rho-\frac12)} = \frac{1}{\pi} \left( I_{r}-I_{\ell}\right),
\end{equation}
where $I_{r}$ and $I_{\ell}$ are the analytic continuation of the integral
\[
I=\frac{1}{2 i} \int \frac{\zeta'(s+\frac12)}{\zeta(s+\frac12)} \frac{z^s}{\sin \pi s} ds,
\]
along the corresponding routes.

\begin{lemma}\label{lema-cota-1} \rm
Let $T>1$, then we have
\begin{align*}\label{zsin-1}
&\left| \frac{z^s}{\sin \pi s} \right| \leq 4 \, e^{\sigma \log |z|} e^{-T(\pi+\arg(z))} \quad \text{if} \quad s=\sigma+iT, \\
&\left| \frac{z^s}{\sin \pi s} \right| \leq 4 \, e^{\sigma \log |z|} e^{-T(\pi-\arg(z))} \quad \text{if} \quad s=\sigma-iT,
\end{align*}
in case that $\sigma>0$ and $|z|<1$ or in case $\sigma<0$ and $|z|>1$.
\end{lemma}
\begin{proof}
\begin{align}
\left| \frac{z^s}{\sin \pi s} \right| &= 2 \left| 
\frac{ e^{(\sigma+iT)(\log|z|+i\arg(z))} }{ e^{i\pi(\sigma+iT)}-e^{-i\pi(\sigma+iT)} } \right| \leq 2 \, \frac{ e^{\sigma \log |z|-T \arg(z)} }{ |e^{-i \pi \sigma} e^{\pi T}|-|e^{i \pi \sigma} e^{-\pi T}| } \nonumber \\ & \leq  2 \, \frac{e^{\sigma \log |z|}e^{-T \arg(z)} }{ e^{\pi T}-e^{-\pi T} } <  4 \, \frac{e^{\sigma \log |z|}e^{-T \arg(z)}}{e^{\pi T}} < 4 \, e^{\sigma \log |z|} e^{-T(\pi+\arg(z))}. \nonumber
\end{align}
The proof for $s=\sigma-iT$ is similar.
\end{proof}

In the following lemma we get bounds of the function $\zeta'(s+1/2)/\zeta(s+1/2)$:
\begin{lemma} \rm
For $\sigma \geq 2$, we have
\[ 
\left| \frac{\zeta'(\sigma+iT)}{\zeta(\sigma+iT)} \right| = \left| \sum_{n=1}^{\infty} \frac{\Lambda(n)}{n^{\sigma+iT}} \right| \leq \sum_{n=1}^{\infty} \left| \frac{\Lambda(n)}{n^{\sigma+iT}} \right|
= \sum_{n=1}^{\infty} \frac{\Lambda(n)}{n^{\sigma}} \leq \sum_{n=1}^{\infty} \frac{\Lambda(n)}{n^2} < 0.57.
\]
For $\sigma < -1$ and $T>1$, using the above bound for $\sigma \geq 2$, the inequalities
\[
\left|\Psi(\sigma+iT)\right| < 3.2 + \frac12 \log (\sigma^2+T^2), \qquad \left|\tan \frac{\pi(\sigma+iT)}{2} \right| < 1.72,
\]
and the functional equation (\ref{zpz}), we get
\[
\left| \frac{\zeta'(\sigma+iT)}{\zeta(\sigma+iT)} \right| < 7.33+\log \sqrt{\sigma^2+T^2} < 7.33 + \log |\sigma| + \log T.
\]
If $-1 < \sigma \leq 2$, then for every real number $T \geq 2$, there exist $T' \in [T, T+1]$ such that uniformly one has 
\[
\left| \frac{\zeta'(\sigma+iT')}{\zeta(\sigma+iT')} \right| < 9 \log^2 T + 2\log T < 11 \log^2 T.
\]
To prove it we first deduce from \cite[Corollary 1]{trudgian} that the number of zeros $\rho$ such that $\gamma \in [T, T+1]$ is less than $ \lfloor 3 \log T \rfloor$. If we subdivide the interval into $1 + \lfloor 3  \log T \rfloor$ equal parts, then the length of each part is $(1+\lfloor 3\log T \rfloor)^{-1}$. As the number of parts exceeds the number of zeros, we deduce applying the Dirichlet pigeon-hole that there is a part that contains no zeros. Hence, for $T'$ lying in this part, we see that
\[
|T' - \gamma | > \frac{1}{1+\lfloor 3\log T \rfloor}.
\]
Hence, we infer that each summand in \cite[Proposition 3.89]{bordelles} is less than $1+\lfloor 3 \log T \rfloor$, and since the number of summands of this kind is less than $\lfloor 3 \log T \rfloor$, we finally get
\[
\left| \frac{\zeta'(\sigma+i T')}{\zeta(\sigma+iT')} \right| < 3 (\log T) \, (1+ 3 \log T).
\]
\end{lemma} 
\noindent Remark: As
\[
\left| \sum_{\gamma \in [T, T+1]} \frac{\sinh x \alpha}{\sinh \pi \alpha} \cos(\alpha \log t) \right| \leq  \sqrt{t} \sum_{\gamma \in [T, T+1]} \exp \frac{-2 \gamma \log T}{T} < 3 \sqrt{t} \, \frac{\log T}{T^2},
\]
the error that we are making in the above left sum when we take $T$ instead of $T'$ is less than $3 \sqrt{t} \, T^{-2} \log T$.

\begin{corollary}\label{lema-cota-integral}
If $x$ and $T$ are related by
\[
\cot \frac{x}{2} = \frac{\log T}{T},
\]
then for $t \geq 2$ and $T \geq 2$, we have
\[
|I_{r}-I_{\ell}| < 44 \, \frac{t^{3/2}+1}{\log t} \frac{\log^2 T}{T^2}.
\]
\end{corollary}
\begin{proof}
As $x$ and $T$ are related by
\[
x=2 \, {\rm arccot} \frac{\log T}{T},
\]
we see that
\[
\pi-x=\mathcal{O}\left( \frac{2 \log T}{T}\right), \qquad e^{-T(\pi-x)} = \mathcal{O}\left(\frac{1}{T^2}\right).
\]
Let $x \in [0, \pi)$, replacing $z$ with $e^{x-i\log t}$, we see that $|z|=t$ and $\arg(z)=x$. Hence
\[
|I_3| < 44 (\log^2 T) \, e^{-T(\pi-x)} \int_{-\frac12}^{\frac32} \, e^{\sigma \log t}  d \sigma + 2.28 \, e^{-T(\pi-x)} \int_{\frac32}^{+\infty} \, e^{\sigma \log t}  d \sigma.  
\]
As we can generalize the integral for $|z|=t>1$ by analytic continuation, for $t \geq 2$ and $T \geq 2$, we get
\[
|I_3| < e^{-T(\pi-x)} \left[ 44 \log^2 T \left( \frac{t^{3/2}}{\log t} - \frac{t^{-1/2}}{\log t} \right) - 2.28 \frac{t^{3/2}}{\log t}  \right].
\]
Hence
\[
|I_3| < 44 \, \frac{t^{3/2}}{\log t} \, \frac{\log^2 T}{T^2}.
\]
For $|I_6|$, we have
\begin{align*}
|I_6| &< 4 e^{-T(\pi-x)} \int_{-\infty}^{-\frac32} \left(7.33 + \log|\sigma|\right) e^{\sigma \log t} d \sigma + 
4 e^{-T(\pi-x)} \log T \int_{-\infty}^{-\frac32} e^{\sigma \log t} d \sigma
\\ & \quad + 44 e^{-T(\pi-x)} \log^2 T \int_{-\frac32}^{-\frac12} e^{\sigma \log t} d \sigma,
\end{align*}
and as $\log |\sigma| < |\sigma|$, and extending the integrals by analytic continuation, for $t \geq 2$ and $T \geq 2$, we get
\[
|I_6| < e^{-T(\pi-x)} \left[\frac{29.4 t^{-3/2}}{\log t} + \frac{6t^{-3/2}}{\log t} -\frac{t^{-3/2}}{\log^2 t} + \frac{4 (\log T) t^{-3/2}}{\log t} + 44(\log^2 T)  \left( \frac{t^{-1/2}}{\log t} -\frac{t^{-3/2}}{\log t} \right) \right].
\]
Hence, for $t \geq 2$ and $T \geq 2$, we have
\[
|I_6| < \frac{\log^2 T}{T^2} \, \frac{44}{\sqrt{t} \log t}
\]
In a similar way we can evaluate the order of $|I_1|$ and $|I_4|$, and we get that they are of order much smaller. 
\end{proof}

\begin{corollary}
For $T \geq 2$ and integers $t \geq 2$, we have
\[
\left| \sum_{\gamma \geq T} \frac{\sinh x \alpha}{\sinh \pi \alpha} \cos(\alpha \log t) \right| < 45 \, \frac{t^{3/2}}{\log t} \frac{\log^2 T}{T^2}.
\]
\end{corollary}
\noindent Compare this bound with that of (\ref{cota}).

\section{On the spectrum of the primes}
The Fourier transform of the Landau formula leads to the following function with peaks at the non-trivial zeros of zeta \cite{Mazur-Stein}:
\[
\Phi_1(t)=-\sum_{m=1}^{T} \frac{\Lambda(m)}{\sqrt{m}} \cos(t \log m).
\]
We have proved the following good approximation for the Mangoldt's function:
\begin{equation}\label{mang-appr}
\frac{\Lambda(t)}{\sqrt{t}} \approx -4\pi \frac{\log T}{T} \sum_{\gamma>0} T^{-2\gamma/T}\cos(\gamma \log t)\, \frac{t^{\mu}+t^{-\mu}}{2} + 2\pi \frac{\log T}{T} \left( \sqrt{t} - \frac{1}{\sqrt{t}(t^2-1)} \right),
\end{equation}
for $T$ sufficiently large, where $\mu=1/2-\beta$, so $-1/2 < \mu < 1/2$. 
Inspired by this formula we construct and study the graphic of the function
\[
\Phi_2(t)=-\sum_{m=1}^{T} T^{-m/T} \frac{\Lambda(m)}{\sqrt{m}} \cos(t \log m) + C \, \sqrt{t},
\]
where $C \approx 0.12$ is a constant. Below we show together two graphics of $\Phi_1(t)$ (in blue) and $\Phi_2(t)$ (in red). We have taken $T=300$.

\begin{figure}[H]
\caption{Peaks at the $\gamma$'s, range 3-30}
\includegraphics[scale=0.75]{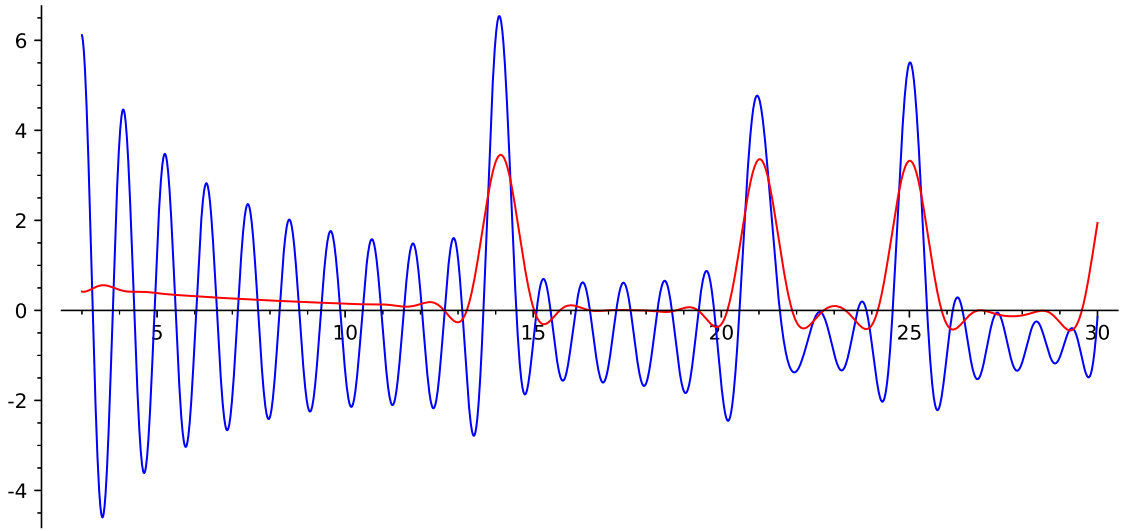}
\label{peaks-zeros-zeta-1}
\end{figure}

\begin{figure}[H]
\caption{Peaks at the $\gamma$'s, range 23-50}
\includegraphics[scale=0.75]{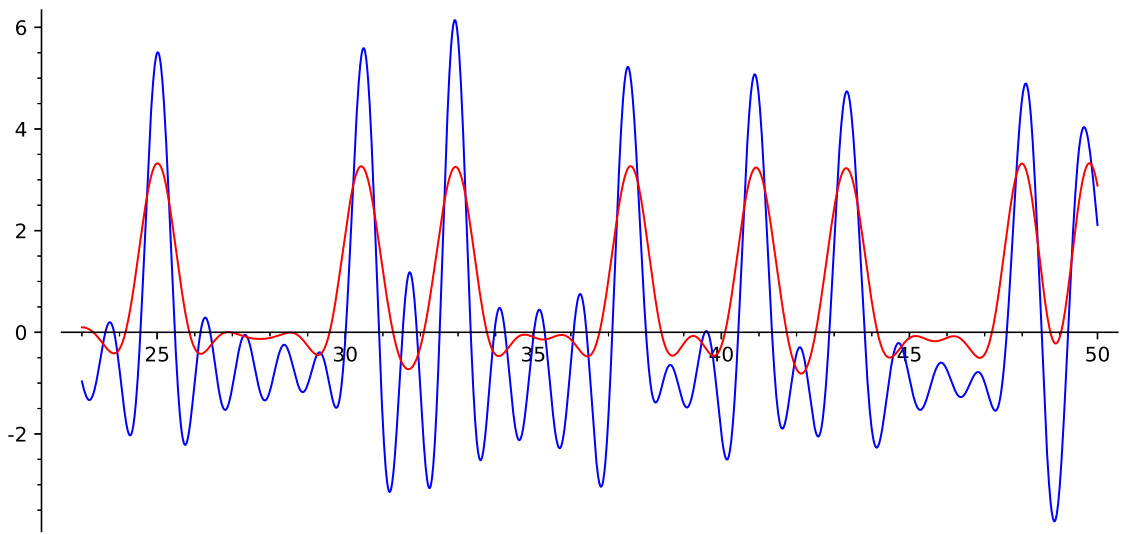}
\label{peaks-zeros-zeta-2}
\end{figure}

We observe that our function $\phi_2(t)$ looks nice. It looks like that this function is interesting and I will continue investigating it.

\section*{Final Remark}
In this paper we have continued our research initiated in \cite{gui} concerning the Mangoldt's function. However this paper is self-contained. In \cite{gui} we also got some new formulas for the Moebius' $\mu$ and Euler's $\varphi$ functions but we only gave the error in the variable $T$ and not its dependence on the variable $t$. In our opinion finding $\varepsilon(T,t)$ could be interesting. This has been done in this paper but only for the Mangoldt's function. In addition we have discovered a new function for the spectrum of the primes, which looks nice.

\section*{Acknowledgements}
Thanks  a lot to Olivier Bordellès for inform me that an upper bound for the number of zeros of zeta such that $\gamma \in [T, T+1]$ can be obtained from \cite[Corollary 1]{trudgian}. Also, many thanks to Juan Arias de Reyna for very interesting comments.

\end{document}